\documentclass[a4paper, 12pt,reqno]{amsart}
\usepackage{amsmath,amssymb,amsthm,enumerate}
\allowdisplaybreaks
\theoremstyle{plain}
\newtheorem{theorem}{Theorem}
\newtheorem*{theorem*}{Theorem}
\newtheorem{lemma}{Lemma}
\newtheorem*{lemma*}{Lemma}
\theoremstyle{definition}
\newtheorem{definition}{Definition}
\newtheorem*{definition*}{Definition}
\theoremstyle{remark}
\newtheorem{remark}{Remark}
\newtheorem*{remark*}{Remark}
\newtheorem{statement}{Statement}
\newtheorem*{statement*}{Statement}

\begin{document}
\title[One distribution function on the Moran sets]{One distribution function on the Moran sets}
\author{Symon Serbenyuk}

\address{
  45~Shchukina St. \\
  Vinnytsia \\
  21012 \\
  Ukraine}
\email{simon6@ukr.net}



\begin{abstract}

In the present article, topological, metric, and fractal properties of certain sets are investigated. These sets are images of sets  whose  elements have restrictions on using  digits or combinations of digits in own s-adic representations, under the map $f$, that is a certain distribution function.

\end{abstract}
\maketitle


\section{Introduction}

Let us consider space $\mathbb R^n$. In \cite{Moran1946}, P. A. P. Moran introduced the following construction of sets and calculated the Hausdorff dimension of the limit set 
\begin{equation}
\label{eq: Cantor-like set}
E=\bigcap^{\infty} _{n=1}{\bigcup_{i_1,\dots , i_n\in A_{0,p}}{\Delta_{i_1i_2\ldots i_n}}}. 
\end{equation}
 Here $p$ is a fixed positive integer, $A_{0,p}=\{1, 2, \dots , p\}$, and sets $\Delta_{i_1i_2\ldots i_n}$ are basic sets  having  the following properties:
\begin{itemize}
\item any set $\Delta_{i_1i_2\ldots i_n}$ is closed and disjoint;
\item for any $i\in A_{0,p}$ the condition $\Delta_{i_1i_2\ldots i_ni}\subset\Delta_{i_1i_2\ldots i_n}$ holds;
\item 
$$
\lim_{n\to\infty}{d\left(\Delta_{i_1i_2\ldots i_n}\right)}=0, \text{where $d(\cdot)$ is the diameter of a set};
$$
\item each basic set is the closure of its interior;
\item at each level the basic sets do not overlap (their interiors are disjoint);
\item any basic set $\Delta_{i_1i_2\ldots i_ni}$ is geometrically similar to $\Delta_{i_1i_2\ldots i_n}$;
\item
$$
\frac{d\left(\Delta_{i_1i_2\ldots i_ni}\right)}{d\left(\Delta_{i_1i_2\ldots i_n}\right)}=\sigma_i,
$$
where $\sigma_i\in (0,1)$ for $i=\overline{1,p}$.
\end{itemize}

The Hausdorff dimension $\alpha_0$ of the set $E$ is the unique root of the following equation
$$
\sum^{p} _{i=1}{\sigma^{\alpha_0} _i}=1.
$$

It is easy to see that set \eqref{eq: Cantor-like set} is a Cantor-like set and a self-similar fractal. The set $E$  is called \emph{the Moran set}.

Much research has been devoted to Moran-like constructions and Cantor-like sets (for example, see \cite{{Falconer1997},{Falconer2004}, {Mandelbrot1977},  {PS1995}, DU2014, DU2014(2),  HRW2000, PS1995,  {S.Serbenyuk 2017}, {S. Serbenyuk fractals}} and references therein).

Fractal sets are widely applicated  in computer design, algorithms of the compression to information, quantum mechanics, solid-state physics, analysis and categorizations of signals of  various forms appearing in different areas (e.g. the analysis of exchange rate fluctuations in economics),  etc. In addition, such sets are useful for checking of preserving the Hausdorff dimension by certain functions \cite{{S. Serbenyuk abstract1},{S.Serbenyuk 2017}}. However, for much classes of fractals the problem of the Hausdorff dimension calculation is difficult and the estimate of parameters  on which the Hausdorff  dimension of certain classes of fractal sets depends is left out of consideration. 

Let $s>1$ be a fixed positive integer. Let us consider the s-adic representation of numbers from~$[0,1]$:
$$
x=\Delta^s _{\alpha_1\alpha_2...\alpha_n...}=\sum^{\infty} _{n=1}{\frac{\alpha_n}{s^n}},
$$
where $\alpha_n\in A=\{0,1,\dots, s-1\}$.

In addition, we say that the following representation
$$
x=\Delta^{-s }_{\alpha_1\alpha_2...\alpha_n...}=\sum^{\infty} _{n=1}{\frac{\alpha_n}{(-s)^n}},
$$
is the nega-s-adic representation of numbers from $\left[-\frac{s}{s+1}, \frac{1}{s+1}\right]$. Here $\alpha_n\in A$ as well.

Some articles (see \cite{ DU2014, DU2014(2), {S. Serbenyuk fractals},{S. Serbenyuk abstract 2}, {S. Serbenyuk abstract 3},{S. Serbenyuk abstract 5}, {Symon1}, {Symon2}, {S. Serbenyuk 2013}, {S. Serbenyuk 2017  fractals}} )  were devoted to   sets whose elements have certain 
restrictions on using combinations of digits in own s-adic  representation. Let us consider the following results.

Suppose  $s>2$ be a fixed positive integer.

Let us consider a class  $\Upsilon_s$ of sets  $\mathbb  S_{(s,u)}$ represented  in the form 
\begin{equation*}
\mathbb S_{(s,u)}= \left\{x: x=\frac{u}{s-1} +\sum^{\infty} _{n=1} {\frac{\alpha_n - u}{s^{\alpha_1+\dots+\alpha_n}}}, (\alpha_n) \in L, \alpha_n \ne u, \alpha_n \ne 0  \right\},
\end{equation*}
where $u=\overline{0,s-1}$, the parameters $u$ and $s$ are fixed for the set $\mathbb  S_{(s,u)}$. That is the class $\Upsilon_s$ contains the sets  $\mathbb  S_{(s,0)}, \mathbb  S_{(s,1)},\dots,\mathbb  S_{(s,s-1)}$. We say that   $\Upsilon$ is a class of sets such that contains the classes   $\Upsilon_3, \Upsilon_4,\dots ,\Upsilon_n,\dots$.

It is easy to see that the set  $\mathbb  S_{(s,u)}$ can be defined by the s-adic representation in the following form
\begin{equation*}
\mathbb S_{(s,u)}=\left\{x: x= \Delta^{s}_{{\underbrace{u\ldots u}_{\alpha_1-1}} \alpha_1{\underbrace{u\ldots u}_{\alpha_2 -1}}\alpha_2 ...{\underbrace{u\ldots u}_{ \alpha_n -1}}\alpha_n...},  (\alpha_n) \in L, \alpha_n \ne u, \alpha_n \ne 0 \right\}, 
\end{equation*}

\begin{theorem}[\cite{{Symon2}, {S. Serbenyuk 2017  fractals}, {S. Serbenyuk   fractals}}]
\label{th: theorem1}
For  an arbitrary $u \in A$ the set $\mathbb S_{(s,u)}$ is an uncountable,   perfect,   nowhere dense set of zero Lebesgue measure, and a self-similar fractal whose Hausdorff dimension $\alpha_0 (\mathbb S_{(s,u)})$ satisfies the following equation 
$$
\sum _{p \ne u, p \in A_0} {\left(\frac{1}{s}\right)^{p \alpha_0}}=1.
$$
\end{theorem}

\begin{remark}
We note that the statement of the last-mentioned theorem is true for all sets $\mathbb  S_{(s,0)}, \mathbb  S_{(s,1)},\dots,\mathbb  S_{(s,s-1)}$ (for  fixed parameters  $u=\overline{0,s-1}$ and any fixed $2<s\in\mathbb N$ )  without the sets $ S_{(3,1)}$ and $ S_{(3,2)}$. 
\end{remark}

\begin{theorem}[\cite{{Symon2}, {S. Serbenyuk 2017  fractals}, {S. Serbenyuk 2013}, {S. Serbenyuk  fractals}}]
\label{th: theorem2}
Let  $E$ be a set, whose elements contain (in own s-adic or nega-s-adic representation) only 
digits or combinations of digits from a certain fixed finite set $\{\sigma_1, \sigma_2,\dots,\sigma_m\}$ of s-adic digits or combinations of digits. 

Then the Hausdorff  dimension $\alpha_0$ of $E$ satisfies the following equation: 
$$
N(\sigma^1 _m)\left(\frac{1}{s}\right)^{\alpha_0}+N(\sigma^2 _m)\left(\frac{1}{s}\right)^{2\alpha_0}+\dots+N(\sigma^{k} _m)\left(\frac{1}{s}\right)^{k\alpha_0}=1,
$$
where $N(\sigma^k_m)$ is a number of k-digit combinations $\sigma^k_m$ from the set $\{\sigma_1, \sigma_2,\dots,\sigma_m\}$,
$k \in \mathbb N$, and $N(\sigma^1 _m)+N(\sigma^2 _m)+\dots+ N(\sigma^{k} _m)=m$.
\end{theorem}

Now we will describe the main function of our investigation. 
Let $\eta$ be a random variable, that defined by the s-adic representation
$$
\eta= \frac{\xi_1}{s}+\frac{\xi_2}{s^2}+\frac{\xi_3}{s^3}+\dots +\frac{\xi_{k}}{s^{k}}+\dots   = \Delta^{s} _{\xi_1\xi_2...\xi_{k}...},
$$
where $\xi_k=\alpha_k$
and digits $\xi_k$ $(k=1,2,3, \dots )$ are random and taking the values $0,1,\dots ,s-1$ with positive probabilities ${p}_{0}, {p}_{1}, \dots , {p}_{s-1}$. That is $\xi_k$ are independent and  $P\{\xi_k=i_k\}={p}_{i_k}$, $i_k \in A$.

From the definition of a distribution function and the following expressions 
$$
\{\eta<x\}=\{\xi_1<\alpha_1(x)\}\cup\{\xi_1=\alpha_1(x),\xi_2<\alpha_2(x)\}\cup \ldots 
$$
$$
\cup\{\xi_1=\alpha_1(x),\xi_2=\alpha_2(x),\dots ,\xi_{k-1}=\alpha_{k-1}(x), \xi_{k}<\alpha_{k}(x)\}\cup \dots,
$$
$$
P\{\xi_1=\alpha_1(x),\xi_2=\alpha_2(x),\dots ,\xi_{k-1}=\alpha_{k-1}(x), \xi_{k}<\alpha_{k}(x)\}
=\beta_{\alpha_{k}(x)}\prod^{k-1} _{j=1} {{p}_{\alpha_{j}(x)}},
$$
where
$$
\beta_{\alpha_k}=\begin{cases}
\sum^{\alpha_k(x)-1} _{i=0} {p_{i}(x)}&\text{whenever $\alpha_k(x)>0$}\\
0&\text{whenever  $\alpha_k(x)=0$,}
\end{cases}
$$
it is easy to see that the following statement is true.

\begin{statement}
The distribution function  ${f}_{\eta}$ of the random variable $\eta$ can be represented in the following form
$$
{f}_{\eta}(x)=\begin{cases}
0&\text{whenever $x< 0$}\\
\beta_{\alpha_1(x)}+\sum^{\infty} _{k=2} {\left({\beta}_{\alpha_k(x)} \prod^{k-1} _{j=1} {{p}_{\alpha_j(x)}}\right)}&\text{whenever $0 \le x<1$}\\
1&\text{whenever $x\ge 1$,}
\end{cases}
$$
where ${p}_{\alpha_{j(x)}}>0$.
\end{statement}

The function
$$ 
{f}(x)=\beta_{\alpha_1(x)}+\sum^{\infty} _{n=2} {\left({\beta}_{\alpha_n(x)}\prod^{n-1} _{j=1} {{p}_{\alpha_j(x)}}\right)},
$$
can be used as a representation of numbers from $[0,1]$. That is
$$
x=\Delta^{P} _{\alpha_1(x)\alpha_2(x)...\alpha_n(x)...}=\beta_{\alpha_1(x)}+\sum^{\infty} _{n=2} {\left({\beta}_{\alpha_n(x)}\prod^{n-1} _{j=1} {{p}_{\alpha_j(x)}}\right)},
$$
where $P=\{p_0,p_1,\dots , p_{s-1}\}$, $p_0+p_1+\dots+p_{s-1}=1$, and $p_i>0$ for all $i=\overline{0,s-1}$. The last-mentioned representation is \emph{the P-representation of numbers from $[0,1]$}.

In the present article, we will consider properties of images of the sets considered in Theorem~\ref{th: theorem1} and Theorem~\ref{th: theorem2} under the map $f$.

We begin with definitions.

Let $s$ be a fixed  positive integer,  $s> 2$. 
Let $c_1, c_2,\dots ,c_m$ be an ordered tuple of integers such that $c_i\in\{0,1,\dots ,s-1\}$ for $i=\overline{1,m}$.

\begin{definition} 
{\itshape A cylinder of rank $m$ with  base $c_1c_2\ldots c_m$} is a set $\Delta^{P} _{c_1c_2\ldots c_m}$ formed by all numbers
of the segment  $[0,1]$ with P-representations in which the first $m$ digits coincide with $c_1,c_2,\dots ,c_m$, respectively, i.e.,
$$
\Delta^{P} _{c_1c_2\ldots c_m}=\left\{x: x=\Delta^{P} _{\alpha_1\alpha_2\ldots\alpha_n\ldots}, \alpha_j=c_j, j=\overline{1,m}\right\}.
$$
\end{definition}

Cylinders $\Delta^{P} _{c_1c_2\ldots c_m}$ have the following properties: 

\begin{enumerate}

\item any cylinder $\Delta^{P} _{c_1c_2\ldots c_m}$ is a closed interval;
\item 
$$
\inf \Delta^{P} _{c_1c_2\ldots c_m}= \Delta^{P} _{c_1c_2\ldots c_m000...},
\sup \Delta^{P} _{c_1c_2\ldots c_m}= \Delta^{P} _{c_1c_2\ldots c_m[s-1][s-1][s-1]...};
$$
\item
$$
| \Delta^{P} _{c_1c_2\ldots c_m}|=p_{c_1}p_{c_2}\cdots p_{c_m}; 
$$
\item
$$
 \Delta^{P} _{c_1c_2\ldots c_mc}\subset  \Delta^{P} _{c_1c_2\ldots c_m};
$$
\item
$$
 \Delta^{P} _{c_1c_2\ldots c_m}=\bigcup^{s-1} _{c=0} { \Delta^{P} _{c_1c_2\ldots c_mc}};
$$
\item
$$
\lim_{m \to \infty} { |\Delta^{P} _{c_1c_2\ldots c_m}|}=0;
$$
\item
$$
\frac{| \Delta^{P} _{c_1c_2\ldots c_mc_{m+1}}|}{| \Delta^{-D} _{c_1c_2\ldots c_m}|}=p_{c_{m+1}};
$$
\item 
$$
\sup\Delta^{P} _{c_1c_2...c_mc}=\inf  \Delta^{P} _{c_1c_2...c_m[c+1]},
$$
where $c \ne s-1$;
\item
$$
\bigcap^{\infty} _{m=1} {\Delta^{-D} _{c_1c_2\ldots c_m}}=x=\Delta^{-D} _{c_1c_2\ldots c_m\ldots}.
$$
\end{enumerate}

\begin{definition} A number $x \in[0,1]$ is called   {\itshape P-rational} if 
$$
x=\Delta^{P} _{\alpha_1\alpha_2\ldots\alpha_{n-1}\alpha_n000\ldots}
$$
or
$$
x=\Delta^{P} _{\alpha_1\alpha_2\ldots\alpha_{n-1}[\alpha_n-1][s-1][s-1][s-1]\ldots}.
$$
The  other  numbers in $[0,1]$ are called {\itshape P-irrational}.
\end{definition}

\section{The objects of research}

Let  $2<s$  be a fixed positive integer, $A=\{0,1,\dots ,s-1\}$, 
$A_0=A \setminus \{0\}=\{1,2,\dots , s -1\}$,  and
$$
 L \equiv  (A_0)^{\infty}= (A_0) \times  (A_0) \times  (A_0)\times\dots  
$$
be the space of one-sided sequences of  elements of $ A_0$.

Let $P=\{p_0,p_1, \dots , p_{s-1}\}$ be a fixed set of positive numbers such that $p_0+p_1+\dots + p_{s-1}=1$.

Let us consider a   class  $\Gamma$ that contains  classes $\Gamma_{P_s}$ of sets  $\mathbb  S_{(P_s,u)}$ represented  in the form 
\begin{equation}
\label{S(s,u)1}
\mathbb S_{(P_s,u)}\equiv\left\{x: x= \Delta^{P}_{{\underbrace{u...u}_{\alpha_1-1}} \alpha_1{\underbrace{u...u}_{\alpha_2 -1}}\alpha_2 ...{\underbrace{u...u}_{ \alpha_n -1}}\alpha_n...},  (\alpha_n) \in L, \alpha_n \ne u, \alpha_n \ne 0 \right\}, 
\end{equation}
where $u=\overline{0,s-1}$, the parameters $u$ and $s$ are fixed for the set $\mathbb  S_{(P_s,u)}$. That is the class $\Gamma_{P_s}$ contains the sets  $\mathbb  S_{(P_s,0)}, \mathbb  S_{(P_s,1)},\dots,\mathbb  S_{(s,s-1)}$.

\begin{lemma}
An arbitrary set $\mathbb  S_{(P_s,u)}$  is a uncountable set.
\end{lemma}
\begin{proof} Let us consider the mapping $g: \mathbb  S_{(P_s,u)} \to  S_u$. That is 
$$
\forall (\alpha_n)\in L: x= \Delta^{P}_{{\underbrace{u...u}_{\alpha_1-1}} \alpha_1{\underbrace{u...u}_{\alpha_2 -1}}\alpha_2 ...{\underbrace{u...u}_{ \alpha_n -1}}\alpha_n...} \stackrel{g}{\longrightarrow}
 \Delta^{s}_{\alpha_1\alpha_2 ...\alpha_n...}=y=g(x).
$$

 It follows from the definition of an arbitrary set  $S_u$ that s-adic-rational numbers of the form 
$$
\Delta^{s} _{\alpha_1\alpha_2\ldots\alpha_{n-1}\alpha_n000\ldots}
$$
 do not belong to $  S_u$ (since the condition $\alpha_n\notin\{0,u\}$ holds). Hence  each element of $  S_u$ has the unique s-adic representation.

For any $x\in \mathbb  S_{(P_s,u)}$ there exists $y=g(x)\in   S_u$ and for any $y\in   S_u$ there exists $x=g^{-1}(y)\in \mathbb  S_{(P_s,u)}$. Since P-rational numbers do not belong to $\mathbb  S_{(P_s,u)}$, we have that   for arbitrary $x_1\ne x_2$ the inequality $f(x_1)\ne f(x_2)$ holds.

So from the  uncountability of $  S_u$ follows the uncountability of the set $\mathbb  S_{(P_s,u)}$.
\end{proof}

To investigate topological and metric properties of  $\mathbb  S_{(P_s,u)}$, we will study properties of cylinders.

Let $c_1, c_2,\dots , c_n$ be an ordered tuple of integers such that $c_i\in\{0,1,\dots ,s-1\}$ for $i=\overline{1,n}$.

\begin{definition} 
{\itshape A cylinder of rank $n$ with  base $c_1c_2\ldots c_n$} is a set $\Delta^{(P,u)} _{c_1c_2\ldots c_n}$ of the form: 
$$
\Delta^{(P,u)} _{c_1c_2\ldots c_n}=\left\{x: x=\Delta^{P}_{{\underbrace{u...u}_{c_1-1}} c_1{\underbrace{u...u}_{c_2 -1}}c_2 ...{\underbrace{u...u}_{ c_n -1}}c_n{\underbrace{u...u}_{\alpha_{n+1}-1}}\alpha_{n+1}{\underbrace{u...u}_{\alpha_{n+2}-1}}\alpha_{n+2}...}, \alpha_j=c_j, j=\overline{1,n}\right\}.
$$
\end{definition}

By $(a_1a_2\ldots a_k)$ denote the period $a_1a_2\ldots a_k$ in the representation of a periodic number.

\begin{lemma} Cylinders $ \Delta^{(P,u)} _{c_1...c_n} $ have the following properties:
\label{lm: Lemma on cylinders}
\begin{enumerate}
\item
$$
\inf  \Delta^{(P,u)} _{c_1...c_n}=\begin{cases}
\Delta^{P} _{{\underbrace{0...0}_{c_1-1}} c_1{\underbrace{0...0}_{c_2 -1}}c_2 ...{\underbrace{0...0}_{ c_n -1}}c_n({\underbrace{0...0}_{ s-2}}[s-1])} &\text{if $u=0$}\\
\Delta^{P} _{{\underbrace{1...1}_{c_1-1}} c_1{\underbrace{1...1}_{c_2 -1}}c_2 ...{\underbrace{1...1}_{ c_n -1}}c_n({\underbrace{1...1}_{ s-2}}[s-1])} &\text{if $u=1$}\\
$$\\
\Delta^{P} _{{\underbrace{u...u}_{c_1-1}} c_1{\underbrace{u...u}_{c_2 -1}}c_2 ...{\underbrace{u...u}_{ c_n -1}}c_n(1)}&\text{if $ u \in \{2,3,\dots ,s-1\}$,}
\end{cases}
$$
$$
\sup  \Delta^{(P,u)} _{c_1...c_n...}=\begin{cases}
\Delta^{P} _{{\underbrace{[s-1]...[s-1]}_{c_1-1}} c_1 ...{\underbrace{[s-1]...[s-1]}_{ c_n -1}}c_n({\underbrace{[s-1]...[s-1]}_{ s-3}}[s-2])} &\text{if $u=s-1$}\\
\Delta^{P} _{{\underbrace{u...u}_{c_1-1}} c_1{\underbrace{u...u}_{c_2 -1}}c_2 ...{\underbrace{u...u}_{ c_n -1}}c_n({\underbrace{u...u}_{ u}}[u+1])} &\text{if $u\in\{1,\dots, s-2\}$}\\
$$\\
\Delta^{P} _{{\underbrace{0...0}_{c_1-1}} c_1{\underbrace{0...0}_{c_2 -1}}c_2 ...{\underbrace{0...0}_{ c_n -1}}c_n(1)}&\text{if $ u=0$.}
\end{cases}
$$
\item If $d(\cdot) $ is the diameter of a set, then
$$
d(\Delta^{(P,u)} _{c_1...c_n})=d(\mathbb S_{(P_s,u)})p^{c_1+c_2+\dots+c_n-n} _{u}\prod^{n} _{j=1}{p_{c_j}}.
$$
\item
$$
\frac{d(\Delta^{(P,u)} _{c_1...c_nc_{n+1}})}{d(\Delta^{(P,u)} _{c_1...c_n})}=p_{c_{n+1}}p^{c_{n+1}-1} _{u}.
$$
\item 
$$
  \Delta^{(P,u)} _{c_1c_2...c_n} =\bigcup^{s-1} _{i=1} { \Delta^{(P,u)} _{c_1c_2...c_ni}}~~~\forall c_n \in A_0,~~~n \in \mathbb N,~ i \ne u.
$$
\item The following relationships are satisfied: 
\begin{enumerate}
\item if $ u\in \{0,1\}$, then 
$$
\inf \Delta^{(P,u)} _{c_1...c_np}> \sup \Delta^{(P,u)} _{c_1...c_n[p+1]};
$$
\item if  $ u \in \{2,3,\dots ,s-3\}$, then 
$$
\begin{cases}
\sup \Delta^{(P,u)} _{c_1...c_np}< \inf \Delta^{(P,u)} _{c_1...c_n[p+1]}&\text{for all $p+1\le u$}\\
$$\\
\inf \Delta^{(P,u)} _{c_1...c_np}> \sup \Delta^{(P,u)} _{c_1...c_n[p+1]},&\text{for all $u<p$;}
\end{cases}
$$
\item if $ u  \in \{s-2,s-1\}$, then
$$
\sup \Delta^{(P,u)} _{c_1...c_np}< \inf \Delta^{(P,u)} _{c_1...c_n[p+1]} ~~~(\text{in this case, the condition $p\ne s-1$ holds}).
$$
\end{enumerate}
\end{enumerate}
\end{lemma}
\begin{proof} \emph{The first property} follows from the equality
$$
x=\Delta^{P}_{{\underbrace{u...u}_{c_1-1}} c_1{\underbrace{u...u}_{c_2 -1}}c_2 ...{\underbrace{u...u}_{ c_n -1}}c_n{\underbrace{u...u}_{\alpha_{n+1}-1}}\alpha_{n+1}{\underbrace{u...u}_{\alpha_{n+2}-1}}\alpha_{n+2}...} 
$$
$$
=\Delta^{P}_{{\underbrace{u...u}_{c_1-1}} c_1{\underbrace{u...u}_{c_2 -1}}c_2 ...{\underbrace{u...u}_{ c_n -1}}c_n(0)}+p^{c_1+\dots +c_n-n} _{u}\left(\prod^{n} _{k=1}{p_{c_k}}\right)\Delta^{P}_{{\underbrace{u...u}_{\alpha_{n+1}-1}}\alpha_{n+1}{\underbrace{u...u}_{\alpha_{n+2}-1}}\alpha_{n+2}...} 
$$
and the definition of  $ \mathbb S_{(P_s,u)}$. 

It is easy to see that \emph{the second property} follows from the first property, \emph{the third property} is a corollary of the first and second properties, and  \emph{Property 4} follows from the definition of the set. 

Let us show that \emph{Property 5} is true. We now prove that the first inequality holds for  $ u=1$. In fact,
$$
\inf \Delta^{(P,0)} _{c_1...c_np}- \sup \Delta^{(P,0)} _{c_1...c_n[p+1]}=
\Delta^{P} _{{\underbrace{0...0}_{c_1-1}} c_1{\underbrace{0...0}_{c_2 -1}}c_2 ...{\underbrace{0...0}_{ c_n -1}}c_n{\underbrace{0...0}_{p -1}}p({\underbrace{0...0}_{ s-2}}[s-1])}-\Delta^{P} _{{\underbrace{0...0}_{c_1-1}} c_1{\underbrace{0...0}_{c_2 -1}}c_2 ...{\underbrace{0...0}_{ c_n -1}}c_n{\underbrace{0...0}_{p}}[p+1](1)}
$$
$$
=\beta_pp^{c_1+...+c_n-n+p-1} _0\prod^{n} _{j=1}{p_{c_j}}+p_pp^{c_1+...+c_n-n+p-1} _0\left(\prod^{n} _{j=1}{p_{c_j}}\right)\inf{\mathbb S_{(P_s,0)}}
$$
$$
-\beta_{p+1}p^{c_1+...+c_n-n+p} _0\prod^{n} _{j=1}{p_{c_j}}-p_{p+1}p^{c_1+...+c_n-n+p} _0\left(\prod^{n} _{j=1}{p_{c_j}}\right)\sup{\mathbb S_{(P_s,0)}}
$$
$$
=p^{c_1+...+c_n-n+p} _0\left(\prod^{n} _{j=1}{p_{c_j}}\right)\left(\beta_pp^{-1} _0+p_pp^{-1} _0\inf{\mathbb S_{(P_s,0)}}-\beta_{p+1}-p_{p+1}\sup{\mathbb S_{(P_s,0)}}\right)
$$
$$
=p^{c_1+...+c_n-n+p-1} _0\left(\prod^{n} _{j=1}{p_{c_j}}\right)(p_0(1-p_0-p_p-p_{p+1}\sup{\mathbb S_{(P_s,0)}})+(1-p_0)(p_1+...+p_{p-1})+p_p\inf{\mathbb S_{(P_s,0)}})>0
$$
because
$$
1-p_0-p_p-p_{p+1}\sup{\mathbb S_{(P_s,0)}}=1-p_0-p_p-p_{p+1}\frac{p_0}{1-p_1}=
\frac{\sum_{i\notin\{0,1,p,p+1\}}p_i+p_{p+1}(1-p_0)+p_0p_1+p_1p_p}{1-p_1}>0.
$$

Also,
$$
\inf \Delta^{(P,1)} _{c_1...c_np}- \sup \Delta^{(P,1)} _{c_1...c_n[p+1]}=
\Delta^{P} _{{\underbrace{1...1}_{c_1-1}} c_1{\underbrace{0...0}_{c_2 -1}}c_2 ...{\underbrace{1...1}_{ c_n -1}}c_n{\underbrace{1...1}_{p -1}}p({\underbrace{1...1}_{ s-2}}[s-1])}-\Delta^{P} _{{\underbrace{1...1}_{c_1-1}} c_1{\underbrace{1...1}_{c_2 -1}}c_2 ...{\underbrace{1...1}_{ c_n -1}}c_n\underbrace{1...1}_{p}[p+1](12)} 
$$
$$
=\beta_pp^{c_1+...+c_n+p-n-1} _1\prod^{n} _{j=1}{p_{c_j}}+p_pp^{c_1+...+c_n-n+p-1} _1\left(\prod^{n} _{j=1}{p_{c_j}}\right)\inf{\mathbb S_{(P_s,1)}}
$$
$$
-\beta_{p+1}p^{c_1+...+c_n+p-n} _1\prod^{n} _{j=1}{p_{c_j}}-p_{p+1}p^{c_1+...+c_n-n+p} _1\left(\prod^{n} _{j=1}{p_{c_j}}\right)\sup{\mathbb S_{(P_s,1)}}
$$
$$
=p^{c_1+...+c_n-n+p-1} _1\left(\prod^{n} _{j=1}{p_{c_j}}\right)\left(\beta_p+p_p\inf{\mathbb S_{(P_s,1)}}-\beta_{p+1}p_1-p_{p+1}p_1\sup{\mathbb S_{(P_s,1)}}\right)
$$
$$
=p^{c_1+...+c_n-n+p-1} _1\left(\prod^{n} _{j=1}{p_{c_j}}\right)(p_p\inf{\mathbb S_{(P_s,1)}}+p_1(1-p_1-p_p-p_{p+1}\sup{\mathbb S_{(P_s,1)}})+(1-p_1)(p_0+p_2+...+p_{p-1}))>0,
$$
since
$$
\sup{\mathbb S_{(P_s,1)}}=\Delta^P _{(12)}=\beta_1+\sum^{\infty} _{k=1}{\beta_1p^k _{1} p^k _2}+\sum^{\infty} _{k=1}{\beta_2p^k _{1} p^k _2}=\frac{p_0+p_0p_1+p^2 _1}{1-p_1p_2}>0
$$
and
$$
1=p_0+p_1+\dots+ p_{s-1}.
$$

Let us prove the system of inequalities. Consider the first inequality. For the case when $p+1\le u$ we get 
$$
\inf \Delta^{(P,u)} _{c_1...c_n[p+1]}-\sup \Delta^{(P,u)} _{c_1...c_np}=\Delta^{P} _{{\underbrace{u...u}_{c_1-1}} c_1{\underbrace{u...u}_{c_2 -1}}c_2 ...{\underbrace{u...u}_{ c_n -1}}c_n{\underbrace{u...u}_{p}} [p+1](1)}-\Delta^{P} _{{\underbrace{u...u}_{c_1-1}} c_1{\underbrace{u...u}_{c_2 -1}}c_2 ...{\underbrace{u...u}_{ c_n -1}}c_n{\underbrace{u...u}_{p-1}} p({\underbrace{u...u}_{ u}}[u+1])}
$$
$$
=\beta_up^{c_1+...+c_n-n+p-1} _u\prod^{n} _{j=1}{p_{c_j}}+\beta_{p+1}p^{c_1+...+c_n-n+p} _u\prod^{n} _{j=1}{p_{c_j}}+p_{p+1}p^{c_1+...+c_n-n+p} _u\left(\prod^{n} _{j=1}{p_{c_j}}\right)\cdot\Delta^P _{(1)}
$$
$$
-\beta_pp^{c_1+...+c_n-n+p-1} _u\prod^{n} _{j=1}{p_{c_j}}-p_{p}p^{c_1+...+c_n-n+p-1} _u\left(\prod^{n} _{j=1}{p_{c_j}}\right)\cdot\Delta^P _{({\underbrace{u...u}_{ u}}[u+1])}
$$
$$
=p^{c_1+...+c_n-n+p-1} _u\left(\prod^{n} _{j=1}{p_{c_j}}\right)\left(\beta_u+\beta_{p+1}p_u+p_{p+1}p_u\Delta^P _{(1)}-\beta_p-p_p\Delta^P _{({\underbrace{u...u}_{ u}}[u+1])}\right)
$$
$$
=p^{c_1+...+c_n-n+p-1} _u\left(\prod^{n} _{j=1}{p_{c_j}}\right)\left(p_{p+1}p_u\Delta^P _{(1)}+(\beta_u-\beta_p)+p_up_p+\beta_pp_u-p_p\Delta^P _{({\underbrace{u...u}_{ u}}[u+1])}\right)>0
$$
since the conditions $p<u$, $\beta_u-\beta_p>0$,   and $\beta_{p+1}=\beta_p+p_{p}$ hold.

Let us prove that the second inequality is true. Here $ p>u $, i.e.,  $p-u \ge 1$. Similarly,
$$
\inf \Delta^{(P,u)} _{c_1...c_np}-\sup \Delta^{(P,u)} _{c_1...c_n[p+1]}=\Delta^{P} _{{\underbrace{u...u}_{c_1-1}} c_1{\underbrace{u...u}_{c_2 -1}}c_2 ...{\underbrace{u...u}_{ c_n -1}}c_n{\underbrace{u...u}_{p-1}}p(1)}-\Delta^{P} _{{\underbrace{u...u}_{c_1-1}} c_1{\underbrace{u...u}_{c_2 -1}}c_2 ...{\underbrace{u...u}_{ c_n -1}}c_n{\underbrace{u...u}_{p}}[p+1]({\underbrace{u...u}_{ u}}[u+1])}
$$
$$
=\beta_pp^{c_1+...+c_n-n+p-1} _u\prod^{n} _{j=1}{p_{c_j}}+p_{p}p^{c_1+...+c_n-n+p-1} _u\left(\prod^{n} _{j=1}{p_{c_j}}\right)\cdot\Delta^P _{(1)}
$$
$$
-\beta_up^{c_1+...+c_n-n+p-1} _u\prod^{n} _{j=1}{p_{c_j}}-\beta_{p+1}p^{c_1+...+c_n-n+p} _u\prod^{n} _{j=1}{p_{c_j}}
-p_{p+1}p^{c_1+...+c_n-n+p} _u\left(\prod^{n} _{j=1}{p_{c_j}}\right)\cdot\Delta^P _{({\underbrace{u...u}_{ u}}[u+1])}
$$
$$
=p^{c_1+...+c_n-n+p-1} _u\left(\prod^{n} _{j=1}{p_{c_j}}\right)\left(\beta_p+p_{p}\Delta^P _{(1)}-\beta_u-\beta_{p+1}p_u-p_{p+1}p_u\Delta^P _{({\underbrace{u...u}_{ u}}[u+1])}\right)
$$
$$
=p^{c_1+...+c_n-n+p-1} _u\left(\prod^{n} _{j=1}{p_{c_j}}\right)\left(p_{p}\Delta^P _{(1)}
+(p_{u+1}+...+p_{p+1})+p_u(p_{p+1}+...+p_{s-1}-p_{p+1}\Delta^P _{({\underbrace{u...u}_{ u}}[u+1])})\right)>0
$$
since the conditions $p>u$, $\beta_p-\beta_u=p_u+p_{u+1}+...+p_{p-1}$, and  $1-\beta_{p+1}=p_{p+1}+...+p_{s-1}$ hold.

Suppose that $u=s-2$. Then 
$$
\inf\Delta^{(P,s-2)} _{c_1c_2...c_n[p+1]}-\sup\Delta^{(P,s-2)} _{c_1c_2...c_np}
$$
$$
=\Delta^{P} _{{\underbrace{[s-2]...[s-2]}_{c_1-1}} c_1{\underbrace{[s-2]...[s-2]}_{c_2 -1}}c_2 ...{\underbrace{[s-2]...[s-2]}_{ c_n -1}}c_n{\underbrace{[s-2]...[s-2]}_{p}} [p+1](1)}
$$
$$
- \Delta^{P} _{{\underbrace{[s-2]...[s-2]}_{c_1-1}} c_1{\underbrace{[s-2]...[s-2]}_{c_2 -1}}c_2 ...{\underbrace{[s-2]...[s-2]}_{ c_n -1}}c_n{\underbrace{[s-2]...[s-2]}_{p-1}} p({\underbrace{[s-2]...[s-2]}_{ s-2}}[s-1])}
$$
$$
=\beta_{s-2}p^{c_1+...+c_n-n+p-1} _{s-2}\prod^{n} _{j=1}{p_{c_j}}+\beta_{p+1}p^{c_1+...+c_n-n+p} _{s-2}\prod^{n} _{j=1}{p_{c_j}}
$$
$$
+p_{p+1}p^{c_1+...+c_n-n+p} _{s-2}\left(\prod^{n} _{j=1}{p_{c_j}}\right)\cdot\Delta^{P} _{(1)}-\beta_pp^{c_1+...+c_n-n+p-1} _{s-2}\prod^{n} _{j=1}{p_{c_j}}
$$
$$
-p_pp^{c_1+...+c_n-n+p-1} _{s-2}\left(\prod^{n} _{j=1}{p_{c_j}}\right)\cdot\Delta^P _{({\underbrace{[s-2]...[s-2]}_{ s-2}}[s-1])}
$$
$$
=p^{c_1+...+c_n-n+p-1} _{s-2}\left(\prod^{n} _{j=1}{p_{c_j}}\right)(\beta_{s-2}+\beta_{p+1}p_{s-2}+p_{s-2}p_{p+1}\Delta^{P} _{(1)}-\beta_{p}-p_p\Delta^P _{({\underbrace{[s-2]...[s-2]}_{ s-2}}[s-1])})
$$
$$
=p^{c_1+...+c_n-n+p-1} _{s-2}\left(\prod^{n} _{j=1}{q_{c_j}}\right)(p_p(1-\Delta^P _{({\underbrace{[s-2]...[s-2]}_{ s-2}}[s-1])})+(p_{p+1}+...+p_{s-3})+\beta_{p+1}p_{s-2}+p_{s-2}p_{p+1}\Delta^{P} _{(1)})>0
$$
since $\beta_{s-2}-\beta_{p}=p_p+p_{p+1}+\dots+p_{s-3}$. Here $p\ne s-1$.

Suppose that $u=s-1$. Then
$$
\inf\Delta^{(P,s-1)} _{c_1c_2...c_n[p+1]}-\sup\Delta^{(P,s-1)} _{c_1c_2...c_np}
$$
$$
=\Delta^{P} _{{\underbrace{[s-1]...[s-1]}_{c_1-1}} c_1{\underbrace{[s-1]...[s-1]}_{c_2 -1}}c_2 ...{\underbrace{[s-1]...[s-1]}_{ c_n -1}}c_n{\underbrace{[s-1]...[s-1]}_{p}} [p+1](1)}
$$
$$
-\Delta^{P} _{{\underbrace{[s-1]...[s-1]}_{c_1-1}} c_1 ...{\underbrace{[s-1]...[s-1]}_{ c_n -1}}c_n{\underbrace{[s-1]...[s-1]}_{p-1}}p({\underbrace{[s-1]...[s-1]}_{ s-3}}[s-2])}
$$
$$
=\beta_{s-1}p^{c_1+...+c_n-n+p-1} _{s-1}\prod^{n} _{j=1}{p_{c_j}}+\beta_{p+1}p^{c_1+...+c_n-n+p} _{s-1}\prod^{n} _{j=1}{p_{c_j}}
$$
$$
+p_{p+1}p^{c_1+...+c_n-n+p} _{s-1}\left(\prod^{n} _{j=1}{p_{c_j}}\right)\cdot\Delta^{P} _{(1)}-\beta_pp^{c_1+...+c_n-n+p-1} _{s-1}\prod^{n} _{j=1}{p_{c_j}}
$$
$$
-p_pp^{c_1+...+c_n-n+p-1} _{s-1}\left(\prod^{n} _{j=1}{p_{c_j}}\right)\cdot\Delta^P _{({\underbrace{[s-1]...[s-1]}_{ s-3}}[s-2])}
$$
$$
=p^{c_1+...+c_n-n+p-1} _{s-1}\left(\prod^{n} _{j=1}{p_{c_j}}\right)(\beta_{s-1}+\beta_{p+1}p_{s-1}+p_{s-1}p_{p+1}\Delta^{P} _{(1)}-\beta_{p}-p_p\Delta^P _{({\underbrace{[s-1]...[s-1]}_{ s-3}}[s-2])})>0.
$$
\end{proof}

\begin{theorem}
The set $\mathbb S_{(P_s,u)}$ is a perfect and nowhere dense set of zero Lebesgue measure.
\end{theorem}
\begin{proof}
 We now prove that \emph{the set $\mathbb  S_{(P_s,u)}$ is a  nowhere dense set}.  From the definition it follows that there exist cylinders $ \Delta^{(P,u)} _{c_1...c_n}$ of rank $n$ in an arbitrary subinterval of the segment    $I=[\inf\mathbb  S_{(P_s,u)},\sup\mathbb  S_{(P_s,u)}]$. Since Property 5 from Lemma~\ref{lm: Lemma on cylinders} is true  for these cylinders, we have that for any subinterval of  $ I$ there exists a subinterval such that does not contain points from $\mathbb  S_{(P_s,u)}$. So $\mathbb  S_{(P_s,u)}$ is a  nowhere dense set.
 
Let us show that \emph{$\mathbb  S_{(P_s,u)}$ is a set of zero Lebesgue measure}. Suppose that $ I^{(P_s,u)} _{c_1c_2...c_n} $ is a closed interval whose endpoints coincide with endpoits of the cylinder $ \Delta^{(P,u)} _{c_1c_2...c_n}$,
$$
|I^{(P_s,u)} _{c_1c_2...c_n}|=d(\Delta^{(P,u)} _{c_1c_2...c_n})=d(\mathbb S_{(P_s,u)})p^{c_1+c_2+\dots+c_n-n} _{u}\prod^{n} _{j=1}{p_{c_j}},
$$
 and
$$
 \mathbb  S_{(P_s,u)}= \bigcap^{\infty} _{k=1} E^{(P_s,u)} _k,
$$
where
$$
E^{(P_s,u)} _1=\bigcup_{c_1\in A_0\setminus\{u\}}{I^{(P_s,u)} _{c_1}},
$$
$$
E^{(P_s,u)} _2=\bigcup_{c_1,c_2\in A_0\setminus\{u\}}{I^{(P_s,u)} _{c_1c_2}},
$$
$$
\dots\dots\dots\dots\dots\dots\dots
$$
$$
E^{(P_s,u)} _k= \bigcup_{c_1,c_2,...,c_k\in A_0\setminus\{u\}}{I^{(P_s,u)} _{c_1c_2...c_k}},
$$
$$
\dots\dots\dots\dots\dots\dots\dots
$$
In addition, since $ E^{(P_s,u)} _{k+1} \subset E^{(P_s,u)} _k $, we have 
$$
E^{(P_s,u)} _k= E^{(P_s,u)} _{k+1} \cup \bar E^{(P_s,u)} _{k+1}.
$$

Let $  I$ be an initial closed interval such that $ \lambda(I)=d_0 $ and $\mathbb  [\inf \mathbb S_{(P_s,u)}, \sup\mathbb S_{(P_s,u)}]=I$, $\lambda(\cdot)$ be the Lebesgue measure of a set. Then
$$
\lambda(E^{(P_s,u)} _1)=\sum_{c_1\in A_0\setminus\{u\}}{|I^{(P_s,u)} _{c_1}|}=d(\mathbb S_{(P_s,u)})\sum_{c_1\in A_0\setminus\{u\}}{p^{c_1-1} _{u}}=\gamma_0.
$$

We get
$$
\lambda(\bar E^{(P_s,u)} _1)=d_0-\lambda(E^{(P_s,u)} _1)=d_0 - \gamma_0 d_0= d_0(1 - \gamma_0).
$$

Similarly,
$$
\lambda(\bar E^{(P_s,u)} _2)=\lambda(E^{(P_s,u)} _1)-\lambda(E^{(P_s,u)} _2)=\gamma_0d_0-\gamma^2 _0d_0=d_0\gamma_0(1-\gamma_0),
$$
$$
\lambda(\bar E^{(P_s,u)} _3)=\lambda(E^{(P_s,u)} _2)-\lambda(E^{(P_s,u)} _3)=\gamma^2 _0d_0-\gamma^3 _0d_0=(1-\gamma_0)\gamma^2 _0d_0,
$$
$$
\dots\dots\dots\dots\dots\dots\dots
$$
So,
$$
\lambda(\mathbb S_{(P_s,u)})=d_0-\sum^{n} _{k=1}{\lambda(\bar E^{(P_s,u)} _k)}=d_0-\sum^{n} _{k=1}{\gamma^{k-1} _0d_0(1-\gamma_0)}=d_0-\frac{d_0(1-\gamma_0)}{1-\gamma_0}=0.
$$
The set $\mathbb  S_{(P_s,u)}$  is a set of zero Lebesgue measure. 

Let us prove that \emph{$\mathbb  S_{(P_s,u)}$  is a perfect set}. Since 
$$ 
E^{(P_s,u)} _k= \bigcup_{c_1,c_2,...,c_k\in A_0\setminus\{u\}}{I^{(P_s,u)}  _{c_1c_2...c_k}}
$$
 is a closed set ($E^{(P_s,u)} _k$ is a union of segments), we see that 
$$
 \mathbb  S_{(P_s,u)}= \bigcap^{\infty} _{k=1} E^{(P_s,u)} _k
$$
is a closed set. 

Let $ x \in \mathbb  S_{(P_s,u)} $,    $ P$  be any interval that contains $ x $, and $ J_n $ be a segment of  $ E^{(P_s,u)} _n $ that contains $ x $. Choose a number $ n $ such that $  J_n \subset P $. Suppose that $ x_n $ is the endpoint of $ J_n $ such that the condition 
$ x_n \ne x $ holds. Hence $ x_n \in \mathbb  S_{(P_s,u)} $ and $  x $ is a limit point of the set. 

Since $\mathbb  S_{(P_s,u)}$ is a closed set and does not contain isolated points, we obtain that $\mathbb  S_{(P_s,u)}$ is a perfect set. 
\end{proof}

\begin{theorem}
The set  $\mathbb S_{(P_s,u)} $ is a self-similar fractal and the Hausdorff dimension $\alpha_0 (\mathbb S_{(P_s,u)})$ of the set satisfies the following equation 
$$
\sum _{i\in A_0\setminus\{u\}} {\left(p_ip^{i-1} _u\right)^{\alpha_0}}=1.
$$
\end{theorem}
\begin{proof}
Since 
$ \mathbb S_{(P_s,u)} \subset I$ and $ \mathbb S_{(P_s,u)}$ is a perfect set, we obtain that  $\mathbb S_{(P_s,u)}$ is a compact set. In addition, 
$$
\mathbb S_{(P_s,u)}=\bigcup_{i\in A_0\setminus\{u\}}{\left[I^{(P_s,u)} _i\cap \mathbb S_{(P_s,u)}\right]}
$$
and $\left[I^{(P_s,u)} _i\cap \mathbb S_{(P_s,u)}\right]\stackrel{p_ip^{i-1} _u}{\sim}\mathbb S_{(P_s,u)}$ for all $i\in A_0\setminus\{u\}$.

Since the set $\mathbb S_{(P_s,u)}$ is a compact self-similar set of space  $ \mathbb R^1 $, we have that the self-similar dimension of this set is equal to the Hausdorff dimension of $\mathbb S_{(P_s,u)}$. So the set  $\mathbb S_{(P_s,u)} $ is a self-similar fractal, and its Hausdorff dimension $\alpha_0$  satisfies the equation
$$
\sum _{i\in A_0\setminus\{u\}} {\left(p_ip^{i-1} _u\right)^{\alpha_0}}=1.
$$
\end{proof}

\begin{theorem}
Let  $E$ be a set whose elements represented in terms of the P-representation by a finite number of fixed combinations $\tau_1, \tau_2,\dots,\tau_m$ of  digits from the alphabet $A$. Then the Hausdorff  dimension $\alpha_0$ of $E$ satisfies the following equation: 
$$
\sum^{m} _{j=1}{\left(\prod^{s-1} _{i=0}{p^{N_i(\tau_j)} _i}\right)^{\alpha_0}}=1,
$$
where $N_i(\tau_k)$ ($k=\overline{1,m}$)  is a number of  the digit $i$ in $\tau_k$ from the set $\{\tau_1, \tau_2,\dots,\tau_m\}$.
\end{theorem}
\begin{proof} Let  $\{\tau_1, \tau_2,\dots,\tau_m\}$ be a set of fixed combinations of  digits from $A$ and the P-representation of any number from $E$ contains only such combinations of digits. 

It is easy to see that there exist combinations $\tau', \tau''$ from the set $\Xi=\{\tau_1, \tau_2,\dots,\tau_m\}$ such that 
$\Delta^P _{\tau^{'}\tau^{'}...}=\inf E$, $\Delta^P _{\tau^{''}\tau^{''}...}=\sup E$,
and
$$
d(E)=\sup E - \inf E=\Delta^P _{\tau^{''}\tau^{''}...}-\Delta^s _{\tau^{'}\tau^{'}...}.
$$

\emph{A cylinder $ \Delta^{(P,E)} _{\tau^{'} _1\tau^{'} _2\ldots\tau^{'} _n}$ of rank $n$ with base $\tau^{'} _1\tau^{'} _2\ldots\tau^{'} _n$} is a set  formed by all numbers of $E$ with the  P-representations in which the first $n$ combinations of digits are fixed and  coincide with $\tau^{'} _1,\tau^{'} _2,\dots,\tau^{'} _n$ respectively ($\tau^{'} _j\in \Xi$ for all $j=\overline{1,n}$).

It is easy to see that
$$
d( \Delta^{(P,E)} _{\tau^{'} _1\tau^{'} _2...\tau^{'} _n})=d(E)\cdot p^{N_0(\tau^{'} _1\tau^{'} _2...\tau^{'} _n)} _0p^{N_1(\tau^{'} _1\tau^{'} _2...\tau^{'} _n)} _1\cdots p^{N_{s-1}(\tau^{'} _1\tau^{'} _2...\tau^{'} _n)} _{s-1}, 
$$
where ${N_i(\tau^{'} _1\tau^{'} _2...\tau^{'} _n)}$ is a number of the digit $i\in A$ in $\tau^{'} _1\tau^{'} _2...\tau^{'} _n$.

Since  $E$ is a closed set, $ E \subset [\inf E, \sup E] $, and
$$
\frac{d\left( \Delta^{(P,E)} _{\tau^{'} _1\tau^{'} _2...\tau^{'} _n\tau^{'} _{n+1}}\right)}{d\left( \Delta^{(P,E)} _{\tau^{'} _1\tau^{'} _2...\tau^{'} _n}\right)}=\prod^{s-1} _{i=0}{p^{N_i(\tau^{'} _{n+1})} _i},
$$
$$
E=[I_{\tau_{1}} \cap E]\cup [I_{\tau_{2}} \cap E]\cup\ldots\cup[I_{\tau_m}\cap E],
$$
where
$I_{\tau_j}=[\inf \Delta^{(P,E)} _{\tau_j},\sup \Delta^{(P,E)} _{\tau_j}]$ and $j=1,2,\dots,m,$
we have
$$
 {[I_{\tau_j} \cap E]} \stackrel{\omega_j}{\sim}E~\text{for all}~j=\overline{1,m},
$$
where 
$$
\omega_j=\prod^{s-1} _{i=0}{p^{N_i(\tau_j)} _i}.
$$
This completes the proof.
\end{proof}

\end{document}